\documentclass[11pt,reqno]{amsart}
\usepackage[dvips]{epsfig}
\RequirePackage{amsmath}
\usepackage{epsfig}
\usepackage{graphicx,epsfig}
\usepackage{amscd,amsthm,amsfonts,amsopn,amssymb,verbatim}
\usepackage{color}
\usepackage{times}
\numberwithin{equation}{section}
\newtheorem{theorem}{Theorem} %[section]

\newtheorem{proposition}[theorem]{Proposition}

\newtheorem{lemma}[theorem]{Lemma}

\theoremstyle{remark}

\def\be{\begin{equation}}
\def\ee{\end{equation}}

\def\ve{\varepsilon}

\allowdisplaybreaks
\def\be{\begin{equation}}
\def\ee{\end{equation}}

\def\ve{\varepsilon}

\def\mod{\text{mod\,}}
\allowdisplaybreaks

\begin{document}
\begin{abstract}
The main purpose of this Note is to provide a non-trivial bound on
certain Kloosterman sums as considered in the recent paper of E.~Fouvry
\cite {F}, leading to a small improvement of his result on the Pell
equation.
\end{abstract}

\title
{A remark on solutions of the Pell equation}
\author
{J.~Bourgain}
\address
{School of Mathematics, Institute for Advanced Study, Princeton, NJ 08540}
\address
{bourgain@math.ias.edu}
%\subjclass{Primary 47A55, Secondary 47L20, 47L30, 47B47.}
%\keywords{normed ideal of operators, essential centre, diagonalization, commutant  mod ideal.}
%\thanks{Research supported in part by NSF Grants DMS 1301619 for first author      and DMS 1001881 for second author.}
\maketitle
\parskip =\medskipamount

\centerline{ 0. INTRODUCTION}

This Note is motivated by the recent work of E.~Fouvry \cite {F}
related to the size of the fundamental solution of the Pell equation and related conjectures due to Hooley \cite{H}.
Let us briefly recall the background.
Let $D$ be a non square positive integer and $\ve_D$ the fundamental solution to the Pell equation
\be\label{0.1}
t^2 -Du^2=1.
\ee 
Following \cite {H} and \cite{F}, introduce for $\alpha>0$ and $x\geq 2$ the set
\be\label{0.2}
S^f(x, \alpha) =|\{(\ve_D, D); 2\leq D\leq x, D \text { non square and $\ve_D\leq D^{\frac 12+\alpha}\}|$}.
\ee
For $0<\ve_0\leq\alpha\leq \frac 12$, Hooley established the asymptotic formula
\be\label{0.3}
S^f(x, \alpha)\sim \frac {4\alpha^2}{\pi^2} x^{\frac 12} (\log x)^2 \text { for } x\to\infty
\ee
(see also Theorem A in \cite {F}).
For $\alpha >\frac 12$, he went on making several further conjectures on the size of this set, in particular the behavior.
\be\label{0.4}
S^f(x, \alpha)\sim B(\alpha) x^{\frac 12} (\log x)^2
\ee
with
\be\label{0.5}
B(\alpha)=
\begin{cases}
&\frac 4{\pi^2} \Big(\alpha -\frac 14\Big) \text { for } \frac 12 <\alpha\leq 1\\
&\frac 4{\pi^2} \Big(\alpha -\frac 14\Big) +\frac 1{18\pi^2}(\alpha - 1)^2 \text { for } 1\leq \alpha \leq \frac 52\\
&\frac 4{\pi^2}\Big(\alpha -\frac 14\Big)+\frac 1{6 \pi^2}\Big(\alpha - \frac 74\Big) \text { for } \alpha >\frac 52.
\end{cases}
\ee
Some of the heuristics in \cite{H} was formalized in \cite{F}, where a lower bound for $\frac 12 \leq \alpha\leq 1$
\be\label{0.6}
S^f(x, \alpha)\geq \frac 1{\pi^2} \big(1+(2\alpha -1) (3-2\alpha)-o(1)\big) x^{\frac 12} (\log x)^2
\ee
is obtained.
In the same paper, Fouvry also shows how to derive that
\be\label{0.7}
S^f(x, \alpha)> \big(B(\alpha)-o(1)\big) x^{\frac 12}(\log x)^2 \ \text { for } \ \frac 12\leq \alpha\leq 1
\ee
assuming a rather modest (but still unproven) bound on certain Kloosterman sums.
The approach in \cite{F} leads indeed to exponential sums of the type
\be\label{0.8}
\sum_{u_1\sim U_1} \beta_{u_1} \sum_{u_2 \sim U_2} e\Big( h\frac {\bar  u_2^2}{u_1^2}\Big)
\ee
where $(u_1, u_2)=1$ and $\bar u_2$ denotes the inverse of $u_2\, \mod\, u_1^2$; $|\beta_{u_1}|\leq 1$.

Fouvry's lower bound \eqref{0.6} relies essentially on various estimates on \eqref{0.8}.
The inner sum is typically an incomplete Kloosterman sum for which presently non-trivial bounds (with power gain) are only available
for $U_2>U_1^{1+\ve}$.
Compared with \eqref{0.7}, the weaker lower bound \eqref{0.6} is due to the exclusion of certain ranges of $U_1, U_2$ for which \eqref{0.8}
cannot be adequately estimated.
In particular, when $\alpha>\frac 12$ is near $\frac 12$, further estimates on \eqref{0.8} for $U_1<U_2<U_1^{1+\ve}$ become quite relevant.

Following an approach initiated by Karacuba and developed further in \cite{B-G1}, \cite{B-G2}, it turns out that one can bound an incomplete
Kloosterman sum
\be\label {0.9}
\sum_{x<N} e\Big(a\frac {\bar x}q\Big)
\ee
non-trivially, for very short intervals $N=q^\rho$ ($\rho>0$ arbitrary, fixed).

It was shown for instance in \cite{B-G2} that
\be\label {0.10}
|\eqref{0.9}|\ll N.(\log q)^{-\frac 12+\ve}.
\ee
This was achieved by decomposing the set $\{1, \ldots, N\}$ in subsets $E$ and $E'$ where $E$ is small and the elements of $E'$ are
`well-factorable' to the extent that one can invoke the theory of multi-linear Kloosterman sums, of the form
\be\label{0.11}
\sum_{x_1\in I_1, \dots, x_r\in I_r} \, e\Big(a\frac {\bar x_1 \bar x_2\cdots \bar x_r}q\Big).
\ee
In \S2 of this paper, we follow a similar path to handle sums
\be\label{0.12}
\sum_{x<N} e\Big(a\frac {\bar x^2}q\Big).
\ee
Given a small parameter $\beta$, we isolate an exceptional subset $E\subset\{1, \ldots, N\}$ of size
\be\label{0.13} 
|E|<\beta\Big(\log \frac 1\beta\Big)^2 N
\ee
(arithmetically defined, independently of the modulus $q$), and such that
\be\nonumber
\sum_{x<N, x\in E} \ e\Big(a\frac {\bar x^2}q\Big)
\ee
can be decomposed in  few multilinear sums
\be\label{0.14}
\sum_{x_1\in I_1, \ldots, x_1\in I_r} \, e\Big(a\frac {\bar x_1^2 \cdots \bar x_r^2}q\Big)
\ee
that may be estimated with a power gain.
It should be noted that in order to achieve this last step, in addition to Karacuba's amplification technique, we also rely 
essentially on \cite {B} \big(which, at least for certain choices of short intervals $I_1, \ldots, I_r$ is not required to handle
\eqref{0.11} \big).

Returning to Fouvry's approach and \eqref{0.8}, we apply the preceding
in the excluded range $U_2<U_1^{1+\ve}$, introducing the exceptional set $E(U_2)$ outside of which the restricted sum
\be\label{0.15}
\sum_{u_2\sim U_2, u_2\not \in E(U_2)} \ e\Big(h\frac {\bar u_2^2}{u_1^2}\Big)
\ee
admits good estimates.
It turns out that the proper choice of $\beta$ is of the form
\be\label{0.16}
\beta= \Big(\alpha -\frac 12\Big)^c
\ee
for some absolute constant $c>0$.
The effect on the main term in Fouvry's approach, by deleting only the subset
\be\label {0.17}
\{(u_1, u_2); u_1\leq x^{\frac 14}, x^{\frac 12} u_1^{-1} \leq u_2\leq \min (x^\alpha u_1^{-1}, x^{\frac 12}u_1), u_2 \in E(U_2)\}
\ee
of the excluded range
\be\label{0.18}
\{(u_1, u_2); u_1\leq x^{\frac 14}, x^{\frac 12} u_1^{-1} \leq u_2\leq \min (x^\alpha u_1^{-1}, x^{\frac 12} u_1)\}
\ee
is to get a lower bound, for $\alpha <\frac 12$ close to $\frac 12$, of the form
\be\label{0.19}
S^f(x, \alpha)> \big(B(\alpha)-\delta(\alpha)\big) x^{\frac 12}(\log x)^2
\ee
where
\be\label{0.20}
\delta (\alpha) =0\Big(\Big(\alpha-\frac 12\Big)^{2+c}\Big)
\ee
rather than
\be\nonumber
\delta(\alpha)= 0\Big(\Big(\alpha -\frac 12\Big)^2\Big)
\ee
resulting from \eqref{0.6}.

\section
{Some Preliminary Lemmas}

\begin{lemma}\label{Lemma1}
Fix some $r\in\mathbb Z_+$ and $\beta>0$ sufficiently small. 
For $n\leq N$, let $n=p_1p_2 \cdots, p_1\geq p_2\geq\cdots$ be its prime
factorization.
Then there is a subset $E\subset \{1, \ldots, N\}$,
\be\label{1.1}
|E|<\beta \Big(\log \frac 1\beta\Big)^r
\ee
such that for $n<N, n\not\in E$, $n$ has at least $r$ prime factors and
\be\label{1.2}
p_r>N^\beta
\ee
\end{lemma}

\begin{proof}
First, the set of integers $n\leq N$ with fewer than $r$ prime factors is bounded by
\be\label{1.3}
C_r\frac { N(\log\log N)^{r-1}}{\log N}.
\ee
We may also assume $p_1\geq N^\alpha$, excluding a set of size 
\be\label{1.4}
\psi(N, N^\alpha)\lesssim \alpha^{\frac 1\alpha} N
\ee
where $\alpha>\beta$ is a parameter to be determined.

Next, assume $n$ of the form
\be\label{1.5}
n=p_1\cdots p_{r_1} n', r_1<r
\ee
$p_1>N^\alpha, p_2, \ldots, p_{r_1}\geq N^\beta$ and $n'$ with no factors larger than $N^\beta$.

Distinguishing the case $n'<y$ and $n'\geq y $ ($y=N^\gamma$ a parameter), the number of those integers may be bounded by
\begin{align}\label{1.6}
&\sum_{n'< y} \ \sum_{p_2, \ldots, p_{r_1} \geq N^\beta} \ \frac N{n' p_2\cdots p_{r_1}} \ \frac 1{\alpha\log N} +
\sum_{\substack {p_1 \cdots p_{r_1}<\frac N y\\ p_1, \ldots, p_{r_1}> N^\beta}} \ \psi \Big(\frac N{p_1\cdots p_{r_1}}, N^\beta\Big)\nonumber\\
&\lesssim \frac N{\alpha\log N}\Big(\log \frac 1\beta\Big)^{r_1-1} \log y +\Big(\frac \beta\gamma\Big)^{\frac \gamma \beta} 
\sum_{\substack {p_1\cdots p_{r_1} <N\\ p_1, \ldots, p_{r_1}>N^\beta}} \ \frac N{p_1\cdots p_{r_1}}\nonumber\\
& \leq \frac\gamma\alpha \Big(\log \frac 1\beta\Big)^{r_1-1} N+\Big(\frac\beta\gamma\Big)^{\frac \gamma\beta} \Big(\log\frac 1\beta\Big)^{r_1} N
\end{align}

Take $\alpha =\Big(\log \frac 1\beta\Big)^{-1}, \gamma =\beta\log \frac 1\beta$.
This gives
\be\label{1.7}
\frac {|E|}N\leq \alpha^{\frac 1\alpha}+ \frac \gamma\alpha \Big(\log \frac 1\beta\Big)^{r-2} +\Big(\frac \beta\gamma\Big)^{\frac \gamma\beta}
\Big(\log \frac 1\beta\Big)^{r-1} \leq \beta\Big(\log \frac 1\beta\Big)^r
\ee
proving Lemma \ref{Lemma1}.
\end{proof}

Next statement is a variant of a Lemma due to Karacuba.

\begin{lemma}\label{Lemma2}
\be\label{1.8}
\Big|\Big\{ (x_1, \ldots, x_{2\ell})\in \mathcal P^{2\ell}; x_i, x_{i+\ell}\sim M_i \text { and }
\frac 1{x_1^2}+\cdots+\frac 1{x_\ell^2} =\frac 1{x^2_{\ell+1}} +\cdots+\frac 1{x^2_{2\ell}}\Big\}\Big|
< (2\ell)^\ell \prod^\ell_{i=1}
\frac {M_i}{\log M_i}
\ee
\end{lemma}

\begin{proof}

If $\frac 1{x_1^2} +\cdots+\frac 1{x_\ell^2} =\frac 1{x^2_{\ell+1}} +\cdots+ \frac 1{x^2_{2\ell}}$, then $x_i|\prod\limits_
{\substack{1\leq j\leq 2\ell\\ j\not= i}} x_j$ and hence \hfill\break $x_i \in \{x_j; j\not= i\}$.
\end{proof}

The statement follows

\section
{A Bound on an Incomplete Kloosterman Sum}

Our aim is to bound the incomplete Kloosterman sum
\be\label{2.1}
\sum_{1\leq x\leq N, (a, q)=1} e_q (a\bar x^2)
\ee
where $q\in \mathbb Z_+, (a, q)=1$ and $\bar x$ is the inverse of $x (\text{mod\,} q)$.

The range $N=q^\rho$ for some fixed $0<\rho<1$.

Let $0<\beta<1$ be a parameter and apply Lemma 1 with $r$ some fixed integer, to be specified later.
The integers $x\in \{1, \ldots, N\}\backslash E$ admit a factorization
\be\label{2.2}
x=p_1\cdots p_r x' \text { with } \ p_1\geq\cdots\geq p_r>p^\beta, \text { prime factors of $x'$ are  $\leq p_r$}.
\ee
We require moreover that
\be\label{2.3}
p_i>\Big(1+\frac {10}{\log N}\Big) p_{i+1} \text { for } 1\leq i\leq r.
\ee
The size of the complementary set is indeed at most
\be\label{2.4}
\sum_{\substack{pp'<N\\ p<p'<p+\frac p{\log N}}} \frac N{pp'} < \sum_{p<N} \frac Np \frac 1{\log N}< \frac {N\log \log N}{\log N}
\ee
Thus the exceptional set $E\subset\{ 1, \ldots, N\}$ satisfies
\be\label{2.5}
|E|< C_r\Big[\beta \Big(\log \frac 1\beta\Big)^r +\frac {\log\log N}{\log N}\Big] N
\ee 
by \eqref{1.1}, \eqref{2.4}.

Partition the remaining integers $n\in \{1, \ldots, N\}\backslash E$ in ranges
$$
\begin{aligned}
\Big(1-\frac 1{\log N}\Big) M_i <p_i & <M_i\qquad (1\leq i\leq r)\\
x'&< \frac {2N}{M_1\cdots M_r}
\end{aligned}
$$
where
\be\label{2.6}
M_i> \Big(1+\frac 2{\log N}\Big) M_{i+1} \, (1\leq i< r)\text { and } \,  M_i> N^\beta\quad (1\leq i\leq r)
\ee
and all prime factor of $x'$ are $<M_r$.
We use here property \eqref{2.3}.

This leads to a bound of the form
\be\label{2.7}
\Big|\sum_{x\leq N, x\not\in E, (x, q) =1} e_q(a\bar x^2)\Big| \lesssim (\log N)^r.\eqref{2.6}
\ee
where \eqref{2.6} denotes an upperbound on a multilinear sum of the type
\be\label{2.8}
\sum_{\substack{p_1\in I_1, \ldots, p_r \in I_r\\ x'\in I_{r+1}}} e_q \big(a\bar p_1^2\cdots \bar p^2_r \bar {x'}^2\big)
\ee
with $I_i$ of the form $I_i =\big[ M_i-\frac {M_i}{\log N}, M_i\big]$, where $M_1, \ldots, M_r>N^\beta$,\hfill\break
 $M_1\ldots M_r|I_{r+1}|<N$.

Fixing $x'\in I_{r+1}$, we bound the sum
\be\label{2.9}
S=\sum_{p_1 \in I_1, \ldots, p_1\in I_r} e_q (a\bar p^2_1\cdots \bar p_r^2)
\ee
using a similar argument as in \cite{B-G2}.

Let $M_i=\frac 12 q^{\beta_i}$.
Hence $\beta_i\geq \rho\beta$.

Let $\ell_1, \ldots, \ell_r \in \mathbb Z_+$ be defined by the condition
\be\label{2.10}
\beta_i(8\ell_i-2)< 1\leq \beta_i(8\ell_i+6).
\ee
It follows that if $x_1, y_1, \ldots x_{2\ell}, y_{2\ell} \in I_i\cap 
\mathcal P$ (coprime with $q$) satisfy
\be\nonumber
\bar x_1^2 +\cdots+\bar x^2_{2\ell}\equiv \bar y_1^2 +\cdots+ \bar y^2_{2\ell} \ (\text {mod\,} q)
\ee
then
\be\nonumber
\Big(\sum^{2\ell}_{i=1} \prod_{j\not= i} x_j^2\Big) \prod_{j=1}^{2\ell} y_j^2 -\Big(\sum^{2\ell}_{i=1} \prod_{j\not= i} y^2_j\Big)
\prod^{2\ell}_{j=1} x_j^2 \equiv 0 \ (\mod q)
\ee
and since this integer is bounded by $2\ell M_i^{8\ell-2}< q$,
\be\nonumber
\frac 1{x_1^2} +\cdots+ \frac 1{x^2_{2\ell}}= \frac 1{y_1^2} +\cdots + \frac 1{y^2_{2\ell}}.
\ee
Estimate using H\"older's inequality and setting $\tilde M_i=\frac {M_i}{\log M_i}$
\be\label{2.11}
|S|^{2^r\ell_1\ldots \ell_r}\leq \frac {(\tilde M_1\cdots \tilde M_r)^{2^r\ell_1\ldots\ell_r}} {\tilde M_1^{2\ell_1}\ldots \tilde M_r^{2\ell_r}}
\quad (2.12)
\ee
with
\be\nonumber
(2.12) =\sum_{z_1, \ldots, z_r} \mu_1 (z_1)\cdots \mu_r (z_r) e_q (a z_1\ldots z_r)
\ee
and
\be\nonumber
\mu_i(z) =|\{ (x_1, \ldots, x_{2\ell_i}) \in (I_i\cap P)^{2\ell_i}; \bar x_1^2 +\cdots +\bar x^2_{\ell_i}
-\cdots - \bar x^2_{2\ell_i}\equiv z(\mod q)\}|,
\ee
From \eqref{2.11} and the preceding, Lemma \ref{Lemma2} implies that
\stepcounter{equation}
\be\label{2.13}
\Vert\mu_i\Vert_2^2 =\sum_z \mu_i (z)^2 < (4\ell)^{2\ell} (\tilde M_i)^{2\ell_i}
\ee
while obviously
\be\label{2.14}
\Vert\mu_i\Vert_1 =(\tilde M_i)^{2\ell_i}.
\ee
Note also that by \eqref{2.10}, certainly $\ell_i\geq\frac 1{14\beta_i}$, hence
\be\label{2.15}
(\tilde M_i)^{2\ell_i} > q^{\frac 18}.
\ee
From \eqref {2.13}, \eqref{2.14}
\be\nonumber
\max \mu(z) =\Vert\mu\Vert_\infty <(4\ell)^\ell \Big(\frac {M_i}{\log M_i}\Big)^\ell < q^{-\frac 18} \Vert\mu\Vert_1.
\ee
The previous argument shows more generally that if $q_1\simeq q^{\ve_1}$ and
\be\label{2.16}
1\geq \ve_1> 6\beta_i
\ee
then
\be\label{2.17}
\max_{\xi \in\mathbb Z/q_1\mathbb Z} \sum_{z\equiv \xi(\text {mod\,} q_1)} \mu(z) < q_1^{-\frac 18} \Vert\mu\Vert_1.
\ee
In order to derive \eqref{2.17}, take $1\leq\ell\leq \ell_i$ such that
\be\nonumber
\beta_i(8\ell-2)< \ve_1 \leq \beta_i(8\ell+6)
\ee
and proceed as above with $x_{\ell+1}, \ldots, x_{\ell_i}, x_{\ell_i+\ell+1},\cdots, x_{2\ell_i}$ frozen.

Thus we define for $z\in \mathbb Z/ q_1\mathbb Z$ the density
\be\nonumber
\mu' (z)=|\{(x_1, \ldots, x_{2\ell})\in (I_i\cap \mathcal P)^{2\ell}; \bar x_1^2 +\cdots +\bar x_\ell^2 -\bar x^2_{\ell+1}
-\cdots- \bar x^2_{2\ell}\equiv z(\text {mod\,} q_1)\}|
\ee
for which again
\be\nonumber
\max_{z\in\mathbb Z/q_1\mathbb Z} \mu'(z) < q_1^{-\frac 18} \Vert\mu'\Vert_1
\ee holds.
Since $\mu$ was disintegrated in such measures $\mu'$, \eqref{2.17} follows.

Our aim is to apply Theorem (**) from \cite{B}.
This result may be reformulated as follows.

\begin{lemma}\label{Lemma3}
Given $\gamma>0$, there is $\ve=\ve(\gamma)>0$, $\tau=\tau(\gamma)>0$ and $k=k(\gamma)\in\mathbb Z_+$ such that the
following holds.

Let $\mu_1, \ldots, \mu_k$ be probability densities on $\mathbb Z/q\mathbb Z$ satisfying
\be\label{2.18}
\max_{\xi\in \mathbb Z/q_1\mathbb Z} \Big[\sum_{Z\equiv \xi(\text{mod\,} q_1)}
\mu_i (z)\Big] < q_1^{-\gamma} \ \text { if } \ q_1|q, q_1> q^\ve.
\ee
Then
\be\label{2.19}
\max_{a\in (\mathbb Z/ q\mathbb Z)^*} \Big|\sum e_q(a x_1\ldots x_k)\mu_1(x_1)\cdots \mu_k(x_k)\Big|< C q^{-\tau}.
\ee
\end{lemma}
\smallskip

In view of \eqref{2.17} we may take $\gamma=\frac 18$.
Note that since $\sum^r_{i=1}\beta_i \sim \rho\leq 1$, we may assume, up to reordering, that $\beta_1, \ldots,
\beta_{[\frac r2]}\leq \frac 2r<\frac 16 \ve\big(\frac 18\big)$ by taking $r$ large enough.
Assume also $r>2k \big(\frac 18\big) $ for Lemma \ref{Lemma3} to apply.

Exploiting $1\leq i\leq [\frac r2]$, we can then satisfy \eqref{2.16} for $\ve_1\geq \ve(\frac 18)$.
As a consequence of \eqref{2.19} one deduces that
\be\label{2.20}
|(2.12)|< Cq^{-\tau(\frac 18)} \ \Vert\mu_1\Vert_1\ldots \Vert\mu_r\Vert_1.
\ee
Recalling \eqref{2.11} and \eqref{2.14}, we proved
that the sum $S$ introduced in \eqref{2.8} satisfies
\be\label{2.21}
|S|< \tilde M_1\cdots \tilde M_r \, q^{-\tau(\frac 18) (2^r\ell_1\ldots\ell_r)^{-1}}.
\ee
By \eqref{2.10}, $\ell_i\sim \frac 1{\beta_i}< \frac 1{\rho\beta} $ for $i=1, \ldots, r$, where $r$ is some constant.
Therefore
\be\label{2.22}
|S|< M_1\cdots M_r \, q^{-c(\rho\beta)^C}
\ee
\be\label{2.23}
|\eqref{2.6}| < N q^{-c(\rho\beta)^C}.
\ee
Finally, substitution in \eqref{2.7} gives
\be\label{2.24}
\Big|\sum_{x\leq N, x\not\in E, (x, q)=1} \ e_q(a\bar x^2)\Big|\lesssim (\log N)^C N^{1-c(\rho\beta)^C}
\ee
where $N=q^\rho$.
In summary, we proved the following
\begin{proposition}

Let $N=q^\rho$.
Given $\frac 1{\log N}< \beta<\frac 1{10}$, there is a subset $E\subset\{1, \ldots, N\}$ (independent of $q$) satisfying
\be\label{2.25}
|E|\lesssim \beta\Big(\log\frac 1\beta\Big)^CN
\ee
and such that for $(a, q)=1$
\be\label{2.26}
\Big|\sum_{x\leq N, x\not\in E, (x, q)=1} \ e_q (a\bar x^2)\Big|\lesssim (\log N)^C N^{1-c(\rho\beta)^C}
\ee
with $c, C>0$ absolute constants.
\end{proposition}

\section
{Fouvry's Approach}

Following \cite{F}, we introduce the sets (for $\alpha>0)$
\be\label{3.1}
S(x, \alpha)=\{(\eta_D, D); 2\leq D\leq x, \text { $D$ nonsquare, $\ve_D\leq \eta_D\leq D^{\frac 12+\alpha}$}\}
\ee
and
\be\label{3.2}
S^f (x, \alpha)=\{(\ve_D, D); 2\leq D\leq x, D \text { nonsquare, $\ve_D\leq D^{\frac 12+\alpha}$}\}
\ee
with $\ve_D$ the fundamental solution of the Pell equation
\be\label{3.3}
t^2 -Du^2=1
\ee
and $\eta_D$ belonging to the set of solutions
\be\label{3.4}
\{\pm \ve_D^n; n\in \mathbb Z\}.
\ee
One has then (of \cite{F}, 17)
\be\label{3.5}
S(x, \alpha)= \sum_{1\leq u\leq X_\alpha} \ \sum_{\Omega\in \mathcal R(u)} \ \sum_{\substack {t\equiv\Omega(\text{\rm mod\,} u^2)\\
Y_2(u, \alpha)\leq t\leq Y_3(u)}} \ 1
\ee
when
\be\label{3.6}
\mathcal R(u) =\{\Omega \, \mod u^2; \Omega^2\equiv 1 (\text{mod\,} u^2)\}
\ee
and denoting
\be\label{3.7}
X_\alpha =\frac 12 (x^\alpha -x^{-1-\alpha})
\ee
\be\label{3.8}
Y_2(u, \alpha)\sim 2^{\frac 1{2\alpha}} u^{1+\frac 1{2\alpha}}, Y_3(u)= u x^{\frac 12}.
\ee
As noted in \cite {F}
\be\label{3.9}
S^f(x, \alpha)= S(x, \alpha) \text { for } \alpha\leq \frac 12
\ee
and
\be\label{3.10}
S(x, \alpha)= S^f(x, \alpha)+ S\Big(x, \frac\alpha2-\frac 12\Big) \text { for } 0\leq \alpha\leq \frac 32.
\ee
Assume $\frac 12 \leq \alpha\leq 1$. Then, following \cite{F}, \S4, we have
\be\label{3.11}
S(x, \alpha)\sim L(x, \alpha)+\frac 1{\pi^2} x^{\frac 12} (\log x)^2
\ee
\be\label{3.12}
L(x, \alpha)=\sum_{X_{\frac 12} <u<X_\alpha} \ \sum_{\substack{Du^2=t^2-1\\ Y_2(u)\leq t\leq Y_3(u)}} \ 1.
\ee

Accounting for non-fundamental solutions is then carried out at the end of \S8 and in \S9 of \cite{F}.
Our aim here is to get a slightly better minoration for \eqref{3.12}, by following Fouvry's argument and complementing
with the Kloosterman sum estimate from \S2.
For reasons of exposition, we first need to recall briefly some of the steps in Fouvry's analysis, referring the reader
to \cite {F} for details.

Assuming $u_1, u_2\geq 1$ coprime, set
\be\label{3.13}
\Phi (u_1, u_2)= -\bar u_1^2 u_1^2 +\bar u_2^2 u_2^2 (\text {mod\,} \hat u)
\ee
with $u=u_1 u_2$, $\bar u_1 (resp. \bar u_2)$ the reciprocals $(\mod u_2^2), (resp. (\text {mod\,} u_1^2)$.

Arithmetical operations permit then to express $L(x, \alpha)$ as a sum
\be\label{3.14}
\sum^\infty_{k=0} \, \sum_{\xi\in\mathcal R(2^k)} L(x, \alpha, \xi, k)
\ee
\be\label{3.15}
L(x, \alpha, \xi, k)=\sum_{\substack{u_1, u_2\\ X_{\frac 12} < 2^k u_1 u_2< X_\alpha\\
(u_1u_2, 2)=(u_1, u_2)=1}} \ 
\sum_{\substack{ t=\Phi (u_1, u_2)\mod u^2_1 u^2_2\\ t\equiv \xi \mod 4^k\\ Y_2(2^k u_1u_2)\leq t\leq Y_3(2^k u_1 u_2)}}
.\ee
The inner sum is further manipulated and expanded in a Fourier series.
The zero-Fourier coefficients in the sum over $u_1, u_2$ contribute to the main term and the other coefficients lead to
error contributions that are captured by trilinear Kloosterman sums of the form
\be\label{3.16}
\sum_{1\leq k\leq H}\alpha_k \sum_{u_1\sim U_1} \beta_{u_1} \sum_{u_2\sim U_2} e\Big(h \frac {\bar u_2^2}{u_1^2}\Big)
\ee
where
\be\label{3.17}
\begin{cases}
U_1\leq U_2\lesssim x^{\frac 12} U_1\\ X_{\frac 12} \leq U_1U_2\lesssim X_\alpha\\ H\ll U_1U_2 x^{-\frac 12+\ve}
\end{cases}
\ee
and $|\alpha_k 1, |\beta_ {u_1}|\leq 1$.

The corresponding main contribution is given by
\be\label{3.18}
EMT (x, U_1, U_2)=\sum_{\substack {u_1, u_2\\ u_1\sim U_1, u_2\sim U_2\\ u_1\equiv \xi_1, u_2=\xi_2 (u^k)\\
(u_1, u_2, 2)=1}} \ \frac { Y_3- Y_2}{u_1^k u_1^2 u_2^2}.
\ee

A main portion of the analysis in \cite {F} consists in bounding \eqref{3.16} in various ranges for $U_1, U_2$, retaining
in \eqref{3.17} only those ranges for which conclusive estimates may be obtained.
Summing the corresponding main terms \eqref{3.18} produces a lower bound on $L(x, \alpha)$ leading to the minoration for
$\frac 12\leq \alpha\leq 1$
\be\label{3.19}
S(x, \alpha)\geq \frac 1{\pi^2} \Big(1+\Big(\alpha-\frac 12\Big)\Big(\frac {11}2-3\alpha\Big)-o(1)\Big)
x^{\frac 12} (\log x)^2
\ee 
as stated in Theorem 1 of \cite{F}.

Using \S2 in this Note, we are able to treat certain additional ranges for $U_1, U_2$ as well, hence narrowing further the
excluded summation range.
This leads to a better minoration for $\alpha>\frac 12$ close to $\frac 12$.

Let us be more precise.
The conclusion from the analysis in \S7, \S8 of \cite {F} is that the `admissible' range for $(u_1, u_2)$ is the set
\be\label{3.20}
\mathcal A=\{(u_1, u_2); u_1\leq x^{\frac 14}, x^{\frac 12} u_1^{-1} \leq u_2\leq \min (x^\alpha u_1^{-1}, x^{\frac 12}
u_1)\}
\ee
leading to a contribution
\be\label{3.21}
8x^{\frac 12} \sum_{\substack{ (u_1, u_2)\in\mathcal A\\ (u_1, u_2)=(u_1u_2, 2) =1}} \frac 1{u_1u_2} -o\big(x^{\frac 12}(\log
x)^2\big)
\ee
to the main term.

Let us consider the range $u_1>x^{\frac 14}$.
Thus, recalling \eqref{3.17}
\be\label{3.22}
\begin{cases}
x^{\frac 14}< u_1<x^{\frac \alpha 2}\\ u_1<u_2<\min (x^\alpha u_1^{-1}, x^{\frac 12} u_1)< x^{\alpha -\frac 14}.
\end{cases}
\ee

Returning to \eqref{3.16}, we apply the Proposition from \S2 to the inner sum with $q=u_1^2$, $N=U_2$.
Hence $\rho\sim \frac 12$.
Let $\beta>0$ be a small parameter (to specify) and $E(U_2)$ the exceptional set obtained.
Thus from \eqref{2.20}, \eqref{2.26}
\be\label{3.23}
|E(U_2)|\lesssim \beta \Big(\log \frac 1p\Big)^C U_2
\ee
while for $u_1\sim U_1$, $h\not= 0$
\be\label{3.24}
\Big|\sum_{\substack{u_2\sim U_2\\ u_2\not\in E}} e\Big( h\frac {\bar u_2^2}{u_1^2}\Big)\Big|
<U_2 \big((h, u_1^2)U_2^{-1}\big)^{c\beta^C}.
\ee
In the range \eqref{3.22}, we only exclude the pairs $(u_1, u_2)$ with $u_2\in E(U_2)$.
Their contribution in the main term is bounded by, cf. \eqref{3.21}
\begin{align}
&O\Big(x^{\frac 12}\sum_{x^{\frac 14} <u_1< x^{\frac \alpha 2}} \frac 1{u_1} \sum_{x^{\frac 14}< U_2<x^{\alpha -\frac 14}}
\ \sum_{u_2\in E(U_2)} \frac 1{u_2}\Big)\nonumber\\
&<O\Big(\Big(\alpha -\frac 12\Big)^2 \beta\Big(\log\frac 1\beta\Big)^C x^{\frac 12} (\log x)^2\Big).
\label{3.25}
\end{align}

The error contributions \eqref{3.16} with $U_1, U_2$ as in \eqref{3.22} need to be modified, leading to expressions
\be\label{3.26}
\sum_{0<h<x^{\alpha -\frac 12+\ve}} \ \sum_{u_1\sim U_1} \Big|\sum _{\substack{ u_2\sim U_2\\ u_2\not\in E(U_2)}}
e\Big(h\frac {\bar u^2_2}{u_1^2}\Big)\Big|
\ee
with suitably restricted $u_2$ variable in the inner sum.
Hence \eqref{3.24} applies and we get
\be\label{3.27}
\eqref{3.26} < x^{\alpha -\frac 12+\ve} \, U_1U_2 (x^{\alpha -\frac 12+\ve}U_2^{-1}) ^{c\beta^C}< x^{2\alpha -\frac 12}
\big(x^{\frac 34-\alpha}\big)^{-c\beta^C}
\ee
Hence we need to require $\alpha>\frac 12$ close enough to $\frac 12$ and take $\beta$ sufficiently large to ensure that
\eqref{3.27} $< x^{\frac 12}$.
Thus $\beta$ is at least a small power of $\alpha-\frac 12$.

In conclusion, incorporating the above in Fouvry's analysis leads to a small improvement in his Theorem 1 for $\alpha$
close to $\frac 12$, to the extent of implying
\be\label{3.28}
S^f(x, \alpha)>\frac 1{\pi^2} \Big(1+4\Big(\alpha-\frac 12\Big)-\delta(\alpha)\Big) x^{\frac 12} (\log x)^2
\ee
where now
\be\label{3.29}
\delta(\alpha)=O\Big(\Big(\alpha -\frac 12\Big)^{2+c}\Big)
\ee
for some $c> 0$.

Recall that
\be\label{3.30}
S^f(x, \alpha)\sim \frac 1{\pi^2} \Big(1+4\Big(\alpha -\frac 12\Big)\Big) x^2(\log x)^2
\ee
corresponds to Hooley's conjecture in the range $\frac 12<\alpha\leq 1$. 
See the discussion in \cite{F}, \S1.

\end{document}